\documentclass[reqno,12pt]{amsart}
\usepackage{color,eucal,enumerate,mathrsfs} %a4wide,
\usepackage[normalem]{ulem}
\usepackage{verbatim}
\usepackage{psfrag}          % replace text in figures
\usepackage{graphicx}
\usepackage{amsmath,amssymb,epsfig}%,bbm}       
\usepackage{tikz}
\usetikzlibrary{patterns,snakes}
\usepackage{xcolor}

\theoremstyle{plain}
\newtheorem{theorem}{Theorem}[section]

\newtheorem{lemma}[theorem]{Lemma}

\theoremstyle{remark}

\newtheorem{question}[theorem]{Question}

\newcommand{\R}{\ensuremath{\mathbb{R}}}

\newcommand{\D}{\ensuremath{\mathcal{D}}}
\newcommand{\N}{\ensuremath{\mathbb{N}}}
\newcommand{\Ll}{\ensuremath{\mathcal{L}}}

\newcommand{\restr}[1]{\lower3pt\hbox{$|_{#1}$}}

\DeclareMathOperator{\graph}{graph}
\DeclareMathOperator{\side}{side}

\begin{document}

\title{A function whose graph has positive doubling measure}

\author{Tuomo Ojala}
\author{Tapio Rajala}
%\author{Ville Suomala}

\address{Department of Mathematics and Statistics \\
         P.O. Box 35 (MaD) \\
         FI-40014 University of Jyv\"askyl\"a \\
         Finland}
\email{tuomo.j.ojala@jyu.fi}
\email{tapio.m.rajala@jyu.fi}
%\email{ville.suomala@jyu.fi}

%\thanks{
%T. Ojala was supported by the Vilho, Yrj\"o and Kalle V\"ais\"al\"a fund,
%T. Rajala by the European Project ERC AdG *GeMeThNES* and 
%V. Suomala by the Academy of Finland project \#126976.
%}
\date{\today}

%%%%%%%%%%%%%%%%%%%%%%%%%%%%%%%%%%%%%%%%%%%%%%%%%%%%%%%%%%%%%%%%%%%%%

\begin{abstract}
  We show that a doubling measure on the plane can give positive measure to the graph of a continuous function. This answers
  a question by Wang, Wen and Wen \cite{WangWenWen2013}. Moreover we show that the doubling constant
  of the measure can be chosen to be arbitrarily close to the doubling constant of the Lebesgue measure.
\end{abstract}

%%%%%%%%%%%%%%%%%%%%%%%%%%%%%%%%%%%%%%%%%%%%%%%%%%%%%%%%%%%%%%%%%%%%%

\maketitle
\section{Introduction}
A Borel regular measure $\mu$ on $\R^d$ is called doubling if there exists 
a constant $C < \infty$ such that
\[
0 < \mu(Q_1) \leq C \mu(Q_2) < \infty,
\]
for any pair of adjacent cubes $Q_1$ and $Q_2$ with the same sidelength. A constant $C$ for which this holds
is called a doubling constant of $\mu$.
In this note we are in dimension two and so the cubes are called squares.
They are understood to be product sets $[x,x+s] \times [y,y+s]$ without rotation, for sidelength $s>0$.
Squares are called adjacent if they intersect.
A set $E \subset \R^d$ is called \emph{thin} if $\mu(E)=0$ for all doubling measures of $\R^d$.
In this paper we answer in negative to a question posed by Wang, Wen and Wen in \cite[Problem 1]{WangWenWen2013}: 
\begin{quote}
Is the $\graph(f):=\left\{ (x,f(x)): x \in [0,1] \right\}$ of a continuous function $f: [0,1] \to [0,1]$ thin? 
\end{quote}
The answer to the aforementioned question is contained in the following
\begin{theorem} \label{thm:continuouscase}
  For any $\epsilon>0$ there exist a doubling measure $\mu$ on $[0,1]^2$ with doubling constant less than $1 + \epsilon$
  and a continuous function $f \colon [0,1] \to [0,1]$ such that 
 \begin{equation}\label{eq:largemeasurecontinuous}
  \mu(\graph(f)) > (1-\epsilon)\mu([0,1]^2).  
 \end{equation}
\end{theorem}

By a density point argument one easily sees that upper-porous sets are thin. Thus the graph of a function $f \colon \R^d \to \R$ 
satisfying
\begin{equation}\label{eq:lip}
  \text{lip}[f](x) = \liminf_{r \to 0} \sup_{y \in B(x,r)}\frac{|f(x)-f(y)|}{r} < \infty
\end{equation}
at every point $x \in \R^d$ is thin.
In particular, the graph of any Lipschitz function is thin. Interestingly, being a graph is to some extent necessary
for this property: Garnett, Killip and Schul showed that there exist rectifiable curves in $\R^d$ that are not thin \cite{GarnettEtal2010}.
Here a rectifiable curve is defined as the image of a continuous curve with finite length.

It would be interesting to further investigate which doubling measures give zero measure to which graphs.
For instance one could study the connection between the doubling constant and the modulus of continuity in this problem.
To our knowledge not much is known on such questions beyond the obviously thin graphs of functions with finite $\text{lip}$
as defined in \eqref{eq:lip},
and the example provided in this note. For instance the following natural question is still open:

%By (cite to Chang-Hao et al. paper) it is known that Bedford-McMullen carpets are thin. Even though there is a 
%connection between H\"older continuous functions and Bedford-McMullen carpets, the following question is still open:
\begin{question}
  Is the graph of a H\"older continuous function $f:[0,1] \to [0,1]$ thin?
\end{question}

\section{Proof of the Theorem}

Let us now prove Theorem \ref{thm:continuouscase}.
By Lusin's theorem, we can reduce the proof to the case of measurable functions. 
Indeed, 
let $\mu$ be a doubling measure and $f \colon [0,1] \to [0,1]$ a measurable function
such that 
\[
 \mu(\graph(f)) =  (1 - \epsilon_1)\mu([0,1]^2) > (1 - \epsilon)\mu([0,1]^2).
\] 
Since the projection $\pi^x_{\sharp}\mu$ of the measure $\mu$ to the real line is a Borel regular 
measure, by Lusin's theorem 
for any $\epsilon_0>0$ we have a continuous function $\hat{f} \colon [0,1] \to [0,1]$ such that 
\[
 \pi^x_{\sharp}\mu (\left\{ x: f(x) \neq \hat{f}(x) \right\})< \epsilon_0 \mu([0,1]^2). 
\]
This implies that
\begin{align*}
\mu(\graph(\hat{f})) & \geq \mu(\graph(f)) - \pi^x_{\sharp}\mu (\left\{ x: f(x) \neq \hat{f}(x) \right\})\\
& >  (1 - \epsilon_1)\mu([0,1]^2) - \epsilon_0 \mu([0,1]^2) >  (1 - \epsilon)\mu([0,1]^2),  
\end{align*}
when $\epsilon_1 + \epsilon_0 \leq \epsilon$.
Thus, to prove Theorem \ref{thm:continuouscase} it is enough to construct 
a doubling measure $\mu$ on $[0,1]^2$ with 
 doubling constant less than $1+\epsilon$
 and a measurable function $f \colon [0,1] \to [0,1]$ such that
 \begin{equation}\label{eq:largemeasure}
  \mu(\graph(f)) > (1-\epsilon)\mu([0,1]^2).  
 \end{equation}

\subsection{Constructing the measure}

  We construct the desired measure using a $4$-adic distribution of mass. 
  The weights used in the distribution will change during the iteration process. 
  
  Take $0 < p<1$ and define $q = 2-p$. These will be fixed throughout the construction and 
  the basic distribution of mass will be done according to these two numbers. The idea of the distribution of mass
  and the respective approximation of the graph after three iteration steps is shown in 
  Figure \ref{fig:const1}. This same pattern would then be repeated in every subrectangle with increased number of steps
  in the $4$-adic division. It is relatively 
  easy to see that while this increasing of steps would result in positive mass for the
  graph, it would also blow up the doubling constant. To ensure that the measure stays
  doubling we have to proceed with a little more care, the main idea being the same.
  \begin{figure}%[h]
    \psfrag{p}{$p$}
    \psfrag{q}{$q$}
    \psfrag{r}{$q^2$}
    \psfrag{s}{$pq$}
    \psfrag{t}{$p^2$}
    \centering
     \includegraphics[width=0.85\textwidth]{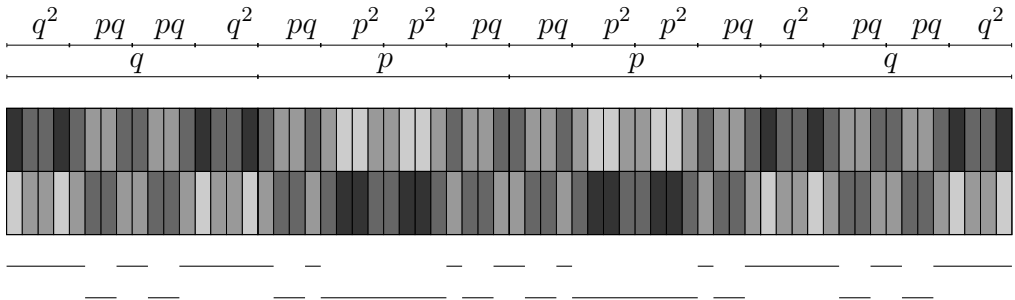}
     \caption{Illustration of the completed second step. Measure inside rectangles is spread using weights $p$ and $q$ in a $4$-adic division.
             This construction with its dual construction (having $p$ and $q$ swapped) placed above each other give the basic idea of the measure.
             In order to obtain a nicely doubling measure the measure has to be corrected near the line where the two constructions meet.
             In constructing the graph we drop from the above and below all the rectangles that have less than avarage measure;
             as well as the areas where the measure had to be redefined. An approximating graph is also illustrated.}
    \label{fig:const1}
  \end{figure}

  Measure $\mu_k$ after each iteration step will be a weighted 
  Lebesgue measure
  \[
   d\mu_k(x,y) = \prod_{i=1}^k w_i(x,y) d\Ll(x,y),
  \]
  where weights $w_i$ are constant in $4$-adic squares of level $i$.
  The weak* limit $\mu$ of the measures $\mu_k$ as $k \to \infty$ will then have the desired properties. Let us now define the weights $w_i$.

  We define the weights $w_i$ using a two step construction. For this purpose we take two sequences $(m_j)_{j = 1}^\infty,
  (n_j)_{j = 1}^\infty$ of positive odd integers larger than $2$ that shall be later specified. Let us call the unit square $[0,1]^2$ the 
  \emph{level $0$ construction rectangle}. For abrivity we denote $M_k = \sum_{j=1}^k(m_j+n_j)$.
  We will define the weights in $4$-adic scales by repeatedly first uniformly distributing in $m_j$ scales and then doing the
  actual redistribution of mass in the following $n_j$ scales.
  
  Let us now give the precise definition for the weights.
  Suppose we are given the level $k$ construction rectangles $[l,r]\times[b,t]$. The weights $w_i$ for 
  steps $M_k < i \le M_{k+1}$ and level $k+1$ construction rectangles inside the rectangles $[l,r]\times[b,t]$
  are defined as follows.
  
  \textbf{First step: uniform distribution.} We define $w_i = 1$ for all steps $M_k < i \le M_k + m_{k+1}$. The purpose for doing this 
  is to have the side length $4^{-M_k - m_{k+1}}$ of the squares compared to the height $t-b$ of the construction rectangle
  to be sufficiently small. As we will see in the second step and in the final computations the $m_k$ will determine the size of the
  set where we move from the weights of the type $p^n$ to the weights of the type $q^n$.
  
  \textbf{Second step: non-uniform distribution.}
  The basic idea of the construction was shown in Figure \ref{fig:const1}, but,
  in order to keep the measure nicely doubling, we only gradually let the
  weights change from $p^n$ to $q^n$ as indicated in Figure \ref{fig:AllConstructionSteps}.
  
  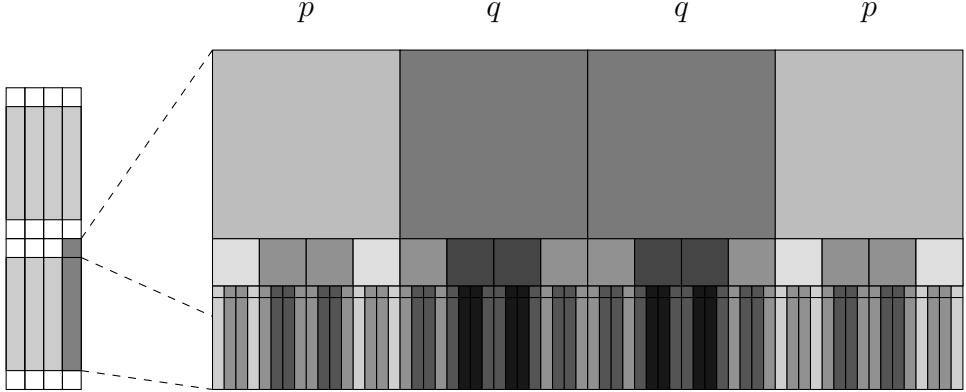
\begin{figure}%[h]
  \begin{centering}
\begin{tikzpicture}
  \pgfmathsetmacro{\bl}{-2} %height of the bottom line
  \pgfmathsetmacro{\cl}{-10/64-10/16} %height of the center line
  \pgfmathsetmacro{\tl}{10/4} %height of the top line
%first the even distribution and cutting the bad parts
    \foreach \i [count=\k]  [evaluate=\k as \y using (\k-1)/4]  in {1,...,4}
    {  
	    \draw [fill=black!20] (\y+5/4-9,-2+1/4) rectangle (\y+5/4+1/4-9,-1/4);
        \draw (\y+5/4-9,-2) rectangle (\y+5/4+1/4-9,0);
	    \draw [fill=black!20] (\y+5/4-9,1/4) rectangle (\y+5/4+1/4-9,2-1/4);
        \draw (\y+5/4-9,0) rectangle (\y+5/4+1/4-9,2);
     }
    
	\draw [fill=black!50] (8/4-9,-2+1/4) rectangle (8/4+1/4-9,-1/4);
	\draw [fill=black!50] (8/4-9,-1/4) rectangle (8/4+1/4-9,0);

%The joining lines indicating where the measure is distributed in which way
    \draw [dashed] (8/4+1/4-9,0) -- (-5,\tl);
    \draw [dashed] (8/4+1/4-9,-1/4) -- (-5,\cl-1/4);
   
    \draw [dashed] (8/4+1/4-9,-2+1/4) -- (-5,\bl);
%Next thing is the measure on the bad parts
    \foreach \i [count=\k]  [evaluate=\k as \y using (\k-1)/4*10-5] [evaluate=\i as \c using (10*\i)]  in {1,4,4,1}
    {  
      \foreach \j [count=\l]  [evaluate=\l as \u using (\l-1)*10/16] [evaluate=\j as \cc using (\c*(\i)+\j*\i/16)-1]  in {1,4,4,1}
      {
	\foreach \n [count=\o]  [evaluate=\o as \z using (\o-1)*10/64] [evaluate=\n as \ccc using (\j+\i+\n)*8-5]  in {1,4,4,1}
	  { 
	    \draw [fill=black!\ccc] (\y+\u+\z,-10/64-10/16) rectangle (\y+10/64+\u+\z,-10/16);
	}
      }
     }

    \foreach \i [count=\k]  [evaluate=\k as \y using (\k-1)/4*10-5] [evaluate=\i as \c using (10*\i)]  in {1,4,4,1}
    {  
      \foreach \j [count=\l]  [evaluate=\l as \u using (\l-1)*10/16] [evaluate=\j as \cc using (\j+\i)*10-7]  in {1,4,4,1}
	{ 
	  \draw [fill=black!\cc] (\y+\u,-10/16) rectangle (\y+10/16+\u,0);
	  \draw [fill=black!\cc] (\y+\u,-10/16) rectangle (\y+10/16+\u,0);
	}
     }
      \foreach \j [count=\l]  [evaluate=\l as \u using (\l-1)*10/4-5] [evaluate=\j as \cc using (13*\j)]  in {2,4,4,2}
	{ \draw [fill=black!\cc] (\u,0) rectangle (10/4+\u,10/4);
      \ifnum \j < 4 
	  \node at (\u+5/4,10/4+.5) {$p$};
	  \else
	  \node at (\u+5/4,10/4+.5) {$q$};
	  \fi
	  }
      %this is the standard part of the construction
    \foreach \i [count=\k]  [evaluate=\k as \y using (\k-1)/4*10-5] [evaluate=\i as \c using (10*\i)]  in {1,4,4,1}
    {  
      \foreach \j [count=\l]  [evaluate=\l as \u using (\l-1)*10/16] [evaluate=\j as \cc using (\c*(\i)+\j*\i/16)-1]  in {1,4,4,1}
      {
	\foreach \n [count=\o]  [evaluate=\o as \z using (\o-1)*10/64] [evaluate=\n as \ccc using (\j+\i+\n)*8-5]  in {1,4,4,1}
	  { 
	    \draw [fill=black!\ccc] (\y+\u+\z,\bl) rectangle (\y+10/64+\u+\z,-10/64-10/16);
	}
      }
     }
\end{tikzpicture}
\end{centering}
 \caption{The weight changes only one step at a time when looking at any vertical line, and the points where the weight
  changes are distances $\frac{1}{4^i}$ apart. On every horizontal line
the weight is divided in $4$-adic manner and the projection to the $y$-axis is the Lebesgue measure.}
\label{fig:AllConstructionSteps}
\end{figure}

%The following process 
%prevents the weight from changing from $p^n$ to $q^n$ too fast.
%We simply stop the iteration gradually at the bottom and top of the upper and lower halves of the construction rectrangle 
%and let it evolve to the desired distribution shown in the 
%Figure \ref{fig:const1} in places where it poses no problems for the doubling property.

In order to obtain this doubling transition at the top and bottom of the construction rectangles 
we define on the upper halves of the construction rectangles of level $k$ living in $[0,1] \times [b, t]$ the weights for
$M_k + m_{k+1} < i \le M_{k+1}$ as
\begin{equation*}
  w_i(x,y) = \left\{ \begin{array}{ll}
    q, & \text{ if } x \in A_i \text{ and } y \in B_i, \\
    p, & \text{ if } x \notin A_i \text{ and } y \in B_i,  \\
    1, & \text{ otherwise, }
\end{array}\right.
\end{equation*}
where
\[
 A_i := \bigcup_{j \in \left\{ 0,1,\dots,4^{(i-1)}-1 \right\} }\left[ \frac{j}{ 4^{(i-1)}} + \frac{1}{4^i} ,\frac{j}{ 4^{(i-1)}} + \frac{3}{4^i}\right]
\]
and
\[
 B_i := \left[\frac{t+b}{2} + \sum_{j=M_k+m_{k+1}+1}^{i-1}\frac{1}{4^{j}}, t - \sum_{j=M_k+m_{k+1}+1}^{i-1}\frac{1}{4^{j}} \right].
\]

On the lower halves of the construction rectangles of level $k$ living in $[0,1] \times [b, t]$ we swap the roles of $p$ and $q$ in defining the weights
$w_i$ for $M_k + m_{k+1} < i \le M_{k+1}$. In other words, we define
\begin{equation*}
  w_i(x,y) = \left\{ \begin{array}{ll}
    p, & \text{ if } x \in A_i \text{ and } \frac{t-b}{2} + y \in {B}_i, \\
    q, & \text{ if } x \notin A_i \text{ and } \frac{t-b}{2} + y \in {B}_i,  \\
    1, & \text{ otherwise, }
\end{array}\right.
\end{equation*}
where $A_i$ and $B_i$ are as above.

After this we define the \emph{construction rectangles of level $k+1$} inside $[0,1] \times [b,t]$ to be the
$4^{-M_{k+1}}$-wide rectangles 
\[
 \left[\frac{j}{4^{M_{k+1}}},\frac{j+1}{4^{M_{k+1}}}\right] \times \left[b+\frac1{4^{M_k+m_{k+1}}},\frac{t+b}{2}-\frac1{4^{M_k+m_{k+1}}}\right]
\]
and
\[
 \left[\frac{j}{4^{M_{k+1}}},\frac{j+1}{4^{M_{k+1}}}\right] \times \left[\frac{t+b}{2}+\frac1{4^{M_k+m_{k+1}}},t-\frac1{4^{M_k+m_{k+1}}}\right]
\]
for $j \in \left\{0,1,\dots,4^{M_{k+1}}-1 \right\}$.

We call the complement of the union of all the construction rectangles of level $k+1$ inside the level $k$ construction rectangles
the \emph{left-over part of level $k+1$}.
On the left-over part of level $k+1$ we define the weight $w_i=1$ for $i>M_{k+1}$. 

\subsection{Approximating the graph}
The next thing is to construct the approximative graph of the desired function. 
We define the approximation of the graph always after the second step is completed. 
That is, the $k$:th approximative graph is defined at the scale $4^{-M_k}$.

Let us denote the first approximative graph by $R_1$.
This will simply be the set of level $1$ construction rectangles where we have more 
weights $q$ than weights $p$ at the level $M_1 = m_1 + n_1$ (see Figure \ref{fig:apprOfGraph}).
In other words, we choose upper or lower construction rectangles depending which of these has larger mass, i.e.
\begin{align*}
  R_1 = \bigg\{ (x,y)\,:\,
    &w_i(x,y) \in \{p, q\} \,\forall i=m_1+1,\dots,m_1+n_1\\
    & \text{ and } \prod_{i = m_1+1}^{m_1+n_1}w_i(x,y) > (pq)^{n_1/2} \bigg\} .
\end{align*}
%One more way to say this is that we either choose upper or lower construction
%rectangle depending on which of these has larger mass.
%which is the union of construction pieces
%with measure at most $p^{(n+1)/2}q^{(n-1)/2}$. 
We can estimate the total measure of the construction rectangles of level $1$ that are not chosen by
\[
 2^{-n_1}\sum_{i=0}^{(n_1-1)/2}\binom{n_1}{i}q^ip^{n_1-i} \le \frac12(pq)^{\frac{n_1-1}{2}}. 
\]
Notice that $pq = 1- (1-p)^2 < 1$. The measure of the left-over part of level $1$ is at most $4^{-m_1+1}$.
%Thus we have for $j\geq m_1+n_1$ 
%\[
%  \mu_j(R_1) \geq 1 - 2^{-n_1}\sum_{i=0}^{(n_1-1)/2}\binom{n_1}{i}q^ip^{n_1-i} - \frac{1}{4^{m_1}} \geq 1 - \frac12(pq)^{\frac{n_1-1}{2}} - \frac{1}{4^{m_1}}.
%\]
% and  $C$ is some constant that is independent of $m$ and $n$ since $p<q$ and (by the Stirling's approximation)
%\[
% \binom{n}{i} \le \binom{n}{(n-1)/2} = \frac{n!}{((n-1)/2)!((n+1)/2)!} \le C \frac{2^n}{\sqrt{n}}.
%\]
Similarly, for $k >1$, let us denote by $R_k$ the union of construction rectangles inside $R_{k-1}$ that have more weights $q$ than weights $p$
in the range $M_{k-1} + m_k < i \leq M_k$. 
The left-over part of level $k$ is included in $2^{k+1}$ strips of width $1$ and of height at most $2^{-k} 4^{-m_k}$.
Therefore the above estimates for the measures of the left-over part
and the construction rectangles
that are not chosen
hold with $n_1$ and $m_1$ replaced by $n_k$ and $m_k$ respectively.
Now we define the $k$:th approximation
of the graph $S_k$ and the actual graph $S$ as
\[
S_k = \bigcap_{j=1}^k R_j \qquad \text{and} \qquad S = \bigcap_{j=1}^\infty R_j.
\]

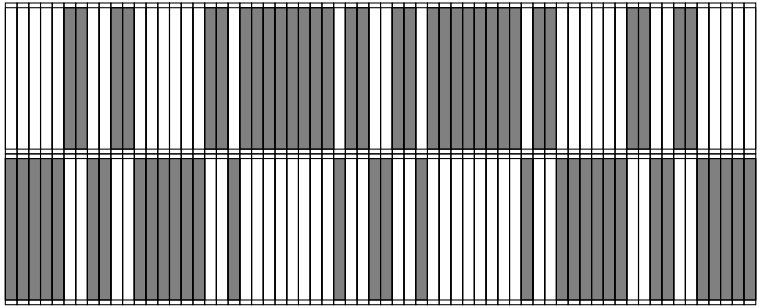
\begin{figure}%[h]
  \begin{center}
    \begin{tikzpicture}
      
    \foreach \i [count=\k]  [evaluate=\k as \y using (\k-1)/4*10-5] [evaluate=\i as \c using (10*\i)]  in {-1,1,1,-1}
    {  
      \foreach \j [count=\l]  [evaluate=\l as \u using (\l-1)*10/16] [evaluate=\j as \cc using \j+\i]  in {-1,1,1,-1}
      {
        \foreach \n [count=\o]  [evaluate=\o as \z using (\o-1)*10/64] [evaluate=\n as \ccc using (\j+\i+\n)]  in {-1,1,1,-1}
	  { 
      \pgfmathparse{\ccc>0?1:0}
      \ifnum\pgfmathresult>0
            \draw [fill=black!50] (\y+\u+\z,-2+4/64+2) rectangle (\y+10/64+\u+\z,-4/64+2);
      \else
            \draw [fill=black!50] (\y+\u+\z,-2+4/64) rectangle (\y+10/64+\u+\z,-4/64);
      \fi
        \draw (\y+\u+\z,-2) rectangle (\y+10/64+\u+\z,0);
        \draw [thin] (\y+\u+\z,4/64) rectangle (\y+10/64+\u+\z,2-4/64);
        \draw [thin] (\y+\u+\z,-4/64) rectangle (\y+10/64+\u+\z,-2+4/64);
        \draw (\y+\u+\z,0) rectangle (\y+10/64+\u+\z,2);
      }
      }
     }
    \end{tikzpicture}
  \end{center}
  \caption{Approximation of the graph is simply the better choice of the two options, 
  at each point $x$ we choose the construction rectangle (either upper or lower) that has more mass.}
  \label{fig:apprOfGraph}
\end{figure}

  It is clear that we get a graph of a function in the limit: when looking at any point $x \in [0,1]$,
  on the line $\left\{ x \right\} \times [0,1]$
  the sequence of the approximations, $S_j \cap \left\{ x \right\} \times [0,1] $, 
  is just a sequence of nested closed
  intervals whose length tends to zero.

  When we choose the sequences $(n_i)$ and $(m_i)$ appropriately we have
  \[
    \mu(S) \ge 1 - \sum_{k=1}^\infty\frac12(pq)^{\frac{n_k-1}{2}} - \sum_{k=1}^\infty\frac{1}{4^{m_k-1}} > 1-\epsilon,
  \]
  so that we have \eqref{eq:largemeasure}.

  By the construction, it is clear that the function is Borel: any pre-image of an open interval is simply a countable union of
  intervals.

\subsection{The measure is doubling}
The only thing left to check is that we actually have a doubling measure as claimed.
For this purpose we rephrase Lemma 2 from \cite{PengWen2014}. 
After this it is easy to check that the construction at hand satisfies the assumptions of the lemma.
\begin{lemma}\cite[Lemma 2]{PengWen2014}
  Let $\D_k$ be a collection of $N$-adic squares of sidelength $N^{-k}$ in $[0,1]^2$. Suppose that $\epsilon>0$ and
  $\left\{ \mu_k \right\}_{k=1}^{\infty}$ is a sequence of probability measures on $[0,1]^2$ satisfying:
  \begin{enumerate}
    \item $\forall k, \; \mu_k \restr{Q} = C_{Q} \Ll$ for all $Q \in \D_k$ where $\Ll$ is Lebesgue measure and
      $C_{Q}>0$ is a constant weight,
    \item $\mu_k(Q) = \mu_{k+1}(Q)$ for all $Q \in \D_k$,
    \item $\mu_k(Q) \leq \epsilon \mu_k(G)$ for $Q, G \in \D_k$ adjacent.
  \end{enumerate}
  Then the sequence $\left\{ \mu_k  \right\}$ converges in a weak$*$-sense 
  to a $C$-doubling measure $\mu$, where $C=C(\epsilon) \to 1$ when $\epsilon \to 1$. 
  \label{le:doubling_limit}
\end{lemma}
% \begin{proof}
%   Proof of the fact that the sequence $\mu_k$ converges to a doubling measure
%   $\mu$ in weak$*$-sense is the same as in \cite{Hanetal2009}.
%   It is also well known that the weak$*$ -limit of $K$ doubling measures is $K$-doubling. 
%   Thus, the only thing left to prove is that each of the measures $\mu_k$ is doubling with a constant $K(\epsilon)$ that
%   tends to one when $\epsilon$ tends to one. 
%   To show this
%   we approximate the measures of given adjacent cubes $Q_1$ and $Q_2$ by the $N$-adic cubes from inside and 
%   from outside. Let us denote $IN(B,n):= \cup \left\{ Q \in \D_n : Q \subset A \right\}$ and 
%   $OUT(B,n) := \cup \left\{ Q \in \D_n : Q \cap B \neq \emptyset \right\}$. 
%   Starting
%   from the largest $N$-adic cube  $Q \in \D_k$ 
%   that fits inside $Q_1$, if we go $m$ steps further in $N$-adic grid, the following two estimates hold:
%   \begin{equation*}
%     \mu_{k+m}(OUT(Q_1,k+m) \setminus IN(Q_1,k+m)) \leq 4 \frac{\epsilon -1}{1 - \epsilon^{-N^m}}\mu_{k+m}(OUT(Q_1,k+m))
%   \end{equation*}
%   and
%   \begin{equation*}
%     \mu_{k+m}(OUT(Q_1,k+m)) \leq \epsilon^{N^m} \mu_{k+m}(OUT(Q_2,k+m)).
%   \end{equation*}
%   From these, the claim follows for $\mu_{m+k}$.
% 
%   We get the same constant for all measures $\mu_j, \; j \geq k+m$, since $\mu_{k+m}(Q) = \mu_j(Q)$ for $Q \in \D_{k+m}$.
%   (Note that for measures $\mu_j$, $j < k+m$ the approximation is even easier, some of the weights just get replaced by $1$.)
% \end{proof}

It is easy to see that our construction satisfies the assumptions of Lemma \ref{le:doubling_limit} with $N=4$. 
Indeed, the first two requirements are clearly satisfied. The third one simply follows from the 
main idea of this type of doubling mass distribution: On adjacent 
squares $Q, G \in \D_k$ the weights $w_j(x,y), \; j\leq k$ can only differ on one index $j$, from where it follows that 
$\mu_k(Q) \leq \frac{q}{p} \mu_k(G)$.

\bibliographystyle{plain}
\bibliography{cubes}
\end{document}